\documentclass[reqno,11pt]{amsart}
\usepackage[utf8]{inputenc}
\DeclareUnicodeCharacter{0301}{}

\calclayout
\usepackage{amsthm,amsfonts,amstext,amssymb,mathrsfs,amsmath,latexsym,mathtools,mathabx}
\usepackage{enumerate}
\usepackage{bm,accents}
\usepackage{mdwlist}
\usepackage[abs]{overpic}
\usepackage[mathscr]{euscript}
\usepackage{mathrsfs}
\DeclareMathAlphabet{\mathpzc}{OT1}{pzc}{m}{it}
\usepackage{parskip}
\usepackage{hyperref}
\usepackage{comment}

\usepackage{tikz}
\usetikzlibrary{shapes.geometric, arrows}

\usepackage[font=small]{caption}
\usepackage[labelformat=simple]{subcaption}

\usepackage{color}
\usepackage{esint} 
\usepackage[colorinlistoftodos,prependcaption,textsize=small]{todonotes}

\definecolor{red}{RGB}{255,0,0}
\definecolor{green}{RGB}{0,100,0}
\definecolor{blue}{RGB}{0,0,255}

\usepackage{cleveref}

\crefname{equation}{equation}{equations}
\crefname{figure}{Figure}{Figures}

\theoremstyle{plain}
\newtheorem{theorem}{Theorem}[section]
 
\newtheorem{lemma}[theorem]{Lemma}

\theoremstyle{definition}

\theoremstyle{remark}
\newtheorem{remark}[theorem]{Remark}
\numberwithin{equation}{section}

\numberwithin{equation}{section}

\renewcommand{\Re}{\mathop{\rm Re}}

\DeclareMathOperator{\supp}{supp}

\newcommand{\N}{\mathbb{N}}
\newcommand{\C}{\mathbb{C}}

\newcommand{\R}{\mathbb{R}}

\newcommand{\isdef}{\overset{\mathrm{def}}{=\joinrel=}}
\renewcommand{\cap}{\mathop{\rm cap}}

\newcommand\restr[2]{{
		\left.\kern-\nulldelimiterspace 
		#1 
		\vphantom{\big|} 
		\right|_{#2} 
}}

\title[Weighted equilibrium and the flow of derivatives of polynomials]{Weighted equilibrium and the flow of derivatives of polynomials}

\dedicatory{Dedicated to Ed Saff on the occasion of his 80th birthday,\\
	in recognition of his profound contributions to approximation theory\\
	and his enduring inspiration as a mentor, colleague, and friend.}
	
\author[A. Mart\'{\i}nez-Finkelshtein]{Andrei Mart\'{\i}nez-Finkelshtein}

\address[AMF]{Department of Mathematics, Baylor University, Waco, TX 76706, USA, and Department of Mathematics, University of Almer\'{\i}a, Almer\'{\i}a, Spain}

\email{A\_Martinez-Finkelshtein@baylor.edu}

\author[E.~A.~Rakhmanov]{Evgenii A.~Rakhmanov}

\address[ER]{Department of Mathematics, University of South Florida, Tampa, FL 33620, USA}

\email{rakhmano@usf.edu}

\date{\today}



\begin{document}

\begin{abstract}
Given a sequence of polynomials $Q_n$ of degree $n$ with zeros on $[-1,1]$, we consider the triangular table of derivatives $Q_{n, k}(x)=d^k Q_n(x) /d x^k$. Under the assumption that the sequence $\{Q_n\}$ has a weak* limiting zero distribution (an empirical distribution of zeros) given by the arcsine law,  we show that as $n, k \rightarrow \infty$ such that $k / n \rightarrow t \in[0,1)$, the zero-counting measure of the polynomials $Q_{n, k}$ converges to an explicitly given measure $\mu_t$. This measure is the equilibrium measure of $[-1,1]$ of size $1-t$ in an external field given by two mass points of size $t/2$ fixed at $\pm 1$. 
The main goal of this paper is to provide a direct potential theory proof of this fact.
\end{abstract}

\maketitle

\section{Introduction} \label{sec:intro}

Given a sequence of (monic) polynomials $Q_n$, $\deg Q_n=n$, $n\in \N$, with zeros on the interval $[-1,1]$, we consider the triangular table of their iterated derivatives 
\begin{equation}
	\label{def:QN}
	Q_{n,k}(x) \isdef  \frac{(n-k)!}{n!}\, \frac{d^k}{dx^k } Q_n(x) = x^{n-k}+\text{lower degree terms}, \quad k=0, 1, \dots, n-1.
\end{equation}
The main result of this paper is the following theorem:
\begin{theorem} \label{mainthm0}
	Let $Q_n$ be a sequence of polynomials, $\deg Q_n=n$, whose zeros belong to the interval $[-1,1]\subset \R$, and such that the sequence of probability zero-counting measures\footnote{\, In all zero-counting measures the zeros are considered with account of their multiplicity.}
	$$
	\sigma_{n,0}= \frac{1}{ n} \chi\left(Q_{n}\right) , \quad \chi(Q_n)\isdef \sum_{Q_n(x)=0} \delta_x,
	$$
	has a weak-* limit $\mu_0$, given by the arc-sine distribution, namely,
	\begin{equation} \label{arcsined}
		d\mu_0(x) = \frac{dx}{\pi \sqrt{1-x^2}}, \quad x\in (-1,1).
	\end{equation}
	Let 
	\begin{equation} \label{eq:defSigmaNK}
		\sigma_{n,k} \isdef  \frac{1}{ n} \chi\left(Q_{n, k }\right) =  \frac{1}{ n} \sum_{Q_{n,k}(x)=0} \delta_x.
	\end{equation}
	
	Then for any $t \in[0,1)$ and any sequence $k_n$ of natural numbers with $k_n / n \rightarrow t$, the sequence $\sigma_{n,k_n}$ has a weak-* limit $\mu_t$, where
	$\mu_t$ is absolutely continuous and
	\begin{equation}\label{eqArcsine}
		d\mu_t(x) = \frac{1}{\pi} \frac{\sqrt{1-t^2-x^2}}{1-x^2}\, dx, \qquad x\in [-\sqrt{1-t^2}, \sqrt{1-t^2}],
	\end{equation}
\end{theorem}

A simple illustration of this result is given by the sequence $Q_n=P_n^{(\alpha, \beta)}$, $\alpha, \beta>-1$, of the (monic) Jacobi polynomials, see \cite[\S18.3]{NIST:DLMF}. The well-known property \cite[\S18.9]{NIST:DLMF} 
$$
Q'_n(x)=n P_{n-1}^{(\alpha+1, \beta+1)}(x)
$$
allows for the explicit computation of the weak-* limit $\mu_t$ of the sequence \eqref{eq:defSigmaNK}, as $k/n\to t$, in this case. This yields the expression in \eqref{eqArcsine}. Theorem~\ref{mainthm0} shows that this behavior is not intrinsic to Jacobi polynomials, but it is generic for  any sequence of polynomials whose initial distribution is given by the arc-sine law. 

Theorem~\ref{mainthm0} can be derived, in particular, from the results contained in the recent paper \cite{MFR2024}, where the description of the Cauchy transform of the limiting measure $\mu_t$ (under very general assumptions) is done in terms of solutions of the inviscid Burgers equation. The main contribution of this paper is the potential-theoretic approach, which characterizes the flow in terms of the equilibrium measure in an external field generated by two mass points.  Since one of the main steps in our analysis is an estimate of the Cauchy integral expressing the polynomials $Q_{n, k} $, it is worth mentioning a recent work of  \cite{Shapiro21} on zeros of polynomials generated by a Rodrigues-type formula.

\section{Proof of Theorem~\ref{mainthm0}}\label{sec:main}
Let us start by setting up the following notation: for $0<\tau\le 1$, denote $ \mathcal M_\tau (F)$ the set of Borel measures of total mass $\tau$ supported on  $F\subset \R$. 

Given a finite compactly supported Borel measure $\mu$ on $\R$, we define its \textit{logarithmic potential} as 
\begin{equation} \label{def:Potential}
	U^\mu(z)\isdef  \int \log \frac{1}{|z-y|}\, d\mu(y). 
\end{equation}
For our purposes, it will be sufficient to consider measures on an interval of the real line, without loss of generality,  $\Delta\isdef [-1,1]\subset \R$. 
Recall the following weighted equilibrium problem that we state in its simplified form, see e.g.~\cite{Saff:97}: given $\tau \in(0,1]$  and a continuous function $\varphi:(-1,1) \to \R$ (the \textit{external field}), there is the unique equilibrium measure $\lambda=\lambda(\tau,\varphi, \Delta) \in \mathcal M_\tau(\Delta )$ minimizing the weighted energy
$$
\int U^{\mu}\, d\mu + 2 \int \varphi \, d\mu
$$
in the class $\mathcal M_\tau(\Delta )$\footnote{\, We could clearly reformulate the problem in the class $\mathcal M_1(\Delta )$; the parameter $\tau$ is introduced here for convenience.}. It is known that there is a constant $m_{\tau ,\varphi}$ such that
\begin{equation} \label{eqCondition}
	U^{\lambda}(x) + \varphi(x) \begin{cases}
		= m_{\tau ,\varphi}, & x\in \supp(\lambda), \\
		\ge m_{\tau ,\varphi}, & x\in \Delta.
	\end{cases}
\end{equation}
The equilibrium measure $\lambda$ is also characterized by the following property:
\begin{equation} \label{eqConstant}
	m_{\tau ,\varphi}=  \min_{x\in \Delta } \left(  U^{\lambda}(x) + \varphi(x) \right) =\max_{\mu\in \mathcal M_\tau(\Delta )} \min_{x\in \Delta } \left(  U^{\mu}(x) + \varphi(x) \right) .
\end{equation}
In other words, for any $\mu \in \mathcal M_\tau(\Delta )$,
\begin{equation} \label{charact1a}
	\min_{x\in \Delta } \left(  U^{\mu}(x) + \varphi(x) \right)  \ge m_{\tau ,\varphi} \text{ quasi-everywhere}\quad \Rightarrow \quad \mu = \lambda(\tau,\varphi, \Delta).
\end{equation}

This will be the key property we will use to prove Theorem~\ref{mainthm0}, taking into account the following well-known fact:
\begin{lemma}\label{lemma:identifyequil}
	For $t\in [0,1)$, measure $\mu_t$ in  \eqref{eqArcsine} is the equilibrium measure $\lambda(1-t,\varphi, \Delta)\in \mathcal M_{1-t}(\Delta)$ in the external field
	\begin{equation} \label{phi}
		\varphi(x)=	\varphi_t(x) \isdef \frac{t}{2} \log \frac{1}{|x^2-1|}, \quad x\in (-1,1).
	\end{equation} 
\end{lemma}
For the proof, see e.g.~\cite[Examples IV.1.18 and IV.6.2]{Saff:97}. 

\begin{lemma}
	Let $\varphi$ be a continuous external field on $\Delta$,  $\tau\in (0,1]$, $\lambda=\lambda(\tau,\varphi, \Delta) \in \mathcal M_\tau(\Delta )$ the equilibrium measure. If the support $S\isdef \supp\left(\lambda_{\tau, \varphi} \right)\subset \Delta $ is an interval then 
	\begin{equation} \label{identityM}
		m_{\tau ,\varphi} = u_\varphi(\infty) + \tau \log \frac{1} {\cap(S)},
	\end{equation}
	where  $u_\varphi$ is the solution of the Dirichlet problem in $\overline{\C} \setminus S$ with the boundary values $\varphi$, and  $\cap(S)=\mathrm{length}(S)/4$ is the logarithmic capacity of $S$. 
	
	In particular, with $\varphi=	\varphi_{t}$ given in \eqref{phi},
	\begin{equation} \label{valueM}
		m_{1-t ,\varphi} =  \log(2)- \frac{1+t}{2}\log(1+t)-  \frac{ 1-t}{2}\log(1-t).
	\end{equation}
\end{lemma}
\begin{proof}
	Let $\Psi$ be the conformal mapping of $\overline{\C} \setminus S$ onto the exterior of the unit circle, such that $\Psi(\infty)=\infty$, and normalized by $\Psi'(\infty)>0$. Then 
	$\Psi'(\infty)=1/\cap(S)$. Consider the function
	$$
	H(z)\isdef U^{\lambda}(z)+\tau \log|\Psi(z)| +u_\varphi(z).
	$$
	By \eqref{eqCondition}, $H$ is harmonic in $\overline{\C}\setminus S$ and identically equal to $m_{\tau ,\varphi}$ on $S$, which implies that $H(z) = m_{\tau ,\varphi}$ for all  $ z\in \overline{\C}\setminus S$. In particular, 
	\begin{align*}
		m_{\tau ,\varphi} & = \lim_{z\to \infty} \left( U^{\lambda}(z)+\tau \log|\Psi(z)| +u_\varphi(z)\right) \\
		& = \lim_{z\to \infty} \left( U^{\lambda}(z)+\tau \log|z|  \right)  +  \lim_{z\to \infty}  \tau \log\left| \frac{\Psi(z)}{z}\right|   +u_\varphi(\infty)
	\end{align*}
	and \eqref{identityM} follows from the well-known properties of the logarithmic potential and the conformal mapping $\Psi$. 
	
	By Lemma~\ref{lemma:identifyequil} and formula \eqref{eqArcsine}, for $\varphi=\varphi_{t}$, we have $\tau=1-t$,  $S=[-s,s]$ with $s=\sqrt{1-t^2}$. In this case, $\Psi(z)=\Phi(z/s)$, where $\Phi(z)$ is the inverse Zhukovski map, 
	\begin{equation} \label{inverseZhuk}
		\Phi(z)=z+(z^2-1)^{1/2}, 
	\end{equation}
	and we take the single-valued branch of the square root in $\C\setminus \Delta$ fixed by $(z^2-1)^{1/2}>0$ for $z>1$. 
	
	On the other hand, 
	$$
	u_\varphi(z)=\Re \left( \frac{\sqrt{z^2-s^2}}{\pi} \int_{-s}^s \frac{\varphi_{t}(x)}{\sqrt{s^2-x^2}}\frac{dx}{z-x}\right),
	$$
	so that
	\begin{align*}
		u_\varphi(\infty) & =  \frac{1}{\pi} \int_{-s}^s \frac{\varphi_{t}(x)}{\sqrt{s^2-x^2}}dx \\
		& = \frac{{t}}{2} \left(  \frac{1}{\pi} \int_{-s}^s \log\frac{1}{1-x} \frac{dx}{\sqrt{s^2-x^2}} +  \frac{1}{\pi} \int_{-s}^s \log\frac{1}{1+x} \frac{dx}{\sqrt{s^2-x^2}}  \right)\\
		& =   \frac{t}{\pi} \int_{-s}^s \log\frac{1}{1-x} \frac{dx}{\sqrt{s^2-x^2}} =t \, U^{\lambda}(1),
	\end{align*}
	where  $\lambda=\lambda(1,0, [-s,s]) $  is the equilibrium (Robin) measure of the interval $[-s,s]$. Since
	$$
	U^{\lambda}(z)=-\log \left(\cap ([-s,s]) \right) -\log\left|\Phi(z/s) \right|= \log \left( \frac{2}{s} \right) -\log\left|\Phi(z/s) \right|,
	$$
	we get that
	$$
	u_\varphi(\infty)=t  \left( \log \left( \frac{2}{s} \right) -  \log\left|\Phi(1/s) \right| \right) = t  \left( \log \left( \frac{2}{s} \right) -  \frac{1}{2}\log( 1+t ) + \frac{1}{2}\log( 1-t )\right) .
	$$
	In consequence,
	\begin{align*}
		m_{\tau ,\varphi}  & = - t   \log\left|\Phi(1/s) \right|   +   \log \left( \frac{2}{s} \right)
		\\
		& = \log(2)- \frac{1+t}{2}\log(1+t)-  \frac{ 1-t}{2}\log(1-t).
	\end{align*}
\end{proof}
\begin{remark}
	Clearly, formula \eqref{identityM} is valid under much milder assumptions on $\varphi$. We have stated the restricted version here, sufficient for our purposes, for the sake of simplicity. 
\end{remark}

Now we turn to the sequence of zero-counting measures $\{	\sigma_{n,k_n}\}_{n\in \Lambda}$, defined in \eqref{eq:defSigmaNK}, with the assumption that
\begin{equation} \label{assumption1}
	\sigma_{n,0} \stackrel{*}{\longrightarrow} \mu_0 \in \mathcal M_1(\Delta), \quad n\to \infty.
\end{equation}
The following lemma is valid even if $\mu_0$ is not the arc-sine distribution on $\Delta$:
\begin{lemma} \label{lem_lowerestimate}
	Let $\{ k_n\}$, $n\in  \N$, be a sequence such that
	\begin{equation} \label{limT}
		\lim_{n  } \frac{k_n}{n} = t\in [0,1)
	\end{equation}
	exists, and let $\Lambda \subseteq \N$ be such that the following weak-* limit exists:
	\begin{equation} \label{assumptNu}
		\sigma_{n,k_n} \stackrel{*}{\longrightarrow} \nu \in \mathcal M_{1-t}(\Delta), \quad n\in \Lambda.
	\end{equation}
	Then, for any $\zeta\in \Delta$ where $U^\nu$ is continuous, and quasi-everywhere on $\Delta$, for any Jordan contour $\Gamma$ enclosing $\Delta$, 
	\begin{equation} \label{eq:estimate}
		U^{ \nu}(\zeta)\ge -c_{t} + \min_{z\in \Gamma} \left( U^{\mu_0}(z)+t\log|z-\zeta| \right),
	\end{equation}
	where
	\begin{equation} \label{defc}
		c_t \isdef t\log t + (1-t) \log(1-t).
	\end{equation}
\end{lemma}
\begin{proof}
	Let first $\zeta $ be in the bounded complement of $\C \setminus (\Gamma \cup \Delta)$. We have the integral representation
	$$
	Q_{n,k}(\zeta)=\binom{n}{k}^{-1} \frac{1}{2\pi i} \oint_\Gamma \frac{Q_n(z)}{(z-\zeta)^{k+1}}dz,
	$$
	from which we get that
	$$
	U^{\sigma_{n,k}}(\zeta)= -\frac{1}{n} \log \left| Q_{n,k}(\zeta)\right| = \frac{1}{n} \log  \binom{n}{k}  +\frac{1}{n}\log(2\pi) - \frac{1}{n} \log \left|  \oint_\Gamma \frac{Q_n(z)}{(z-\zeta)^{k+1}}dz \right|.
	$$
	By Stirling's formula, 
	$$
	\lim_n \frac{1}{n} \log  \binom{n}{k_n} = -c_t.
	$$
	On the other hand,
	\begin{align*}
		\frac{1}{n} \log \left|  \oint_\Gamma \frac{Q_n(z)}{(z-\zeta)^{k+1}}dz \right| & \le \frac{1}{n} \log \left(  \max_{z\in \Gamma} \left| \frac{Q_n(z)}{(z-\zeta)^{k+1}}\right| \cdot \text{length}(\Gamma) \right)\\
		&=\max_{z\in \Gamma} \left(-U^{\sigma_{n,0}}(z)-\frac{k+1}{n} \log|z-\zeta| \right)+\frac{1}{n}   \log\left(  \text{length}(\Gamma)\right)\\
		&=-\min_{z\in \Gamma} \left(U^{\sigma_{n,0}}(z)+\frac{k+1}{n} \log|z-\zeta| \right)+\frac{1}{n}   \log\left(  \text{length}(\Gamma)\right).
	\end{align*}
	Thus, we conclude that
	\begin{equation} \label{lowerbd2}
		U^{\sigma_{n,k_n}}(\zeta)\ge -c_t + \min_{z\in \Gamma} \left(U^{\sigma_{n,0}}(z)+\frac{k_n+1}{n} \log|z-\zeta| \right)+o(1), \quad n\to \infty.
	\end{equation}
	Taking limits in both sides and using assumptions \eqref{assumption1} and \eqref{assumptNu} we conclude that \eqref{eq:estimate} is valid for $\zeta $   in the bounded complement of $\C \setminus (\Gamma \cup \Delta)$. 
	
	By the the Continuity Theorem \cite[Theorem A.IV.4.6]{Saff:97}, we can extend this inequality to any $\zeta \in\Delta$ where the potential $U^\nu$ is continuous. Moreover, by the Lower Envelope Theorem  \cite[Theorem A.IV.4.8]{Saff:97}, it holds quasi-everywhere on $\Delta$.
\end{proof}

We need to estimate the right-hand side of \eqref{eq:estimate} in the case when $\mu_0$ is the arcsine distribution given in \eqref{arcsined}. In this case,
\begin{equation} \label{equilPot}
	U^{\mu_0}(z) = \log(2) - \log \left| \Phi(z) \right|,
\end{equation}
with $\Phi$ given in \eqref{inverseZhuk}. 

Also, following \cite{Saff:97}, we use the notation
$$
\text{``}\inf_{z \in \Delta}\text{''} h(z)
$$
for essential infimum, that is, the largest number $L$ such that on $\Delta$ the real function $h$ takes values smaller than $L$ only on a set of zero capacity.
\begin{lemma}\label{lem_lowerestimateArcsin}
	With the assumptions of Lemma~\ref{lem_lowerestimate},
	\begin{equation} \label{eq:newbound}
		\text{``}\inf_{x\in [-1,1]}\text{''}	\left( U^{\nu}+\varphi\right) (x)  \ge   	m_{1-t ,\varphi} ,
	\end{equation}
	where $m_{1-t ,\varphi}$ is given in \eqref{valueM}. 
\end{lemma}
\begin{proof}
	Fix $\zeta\in (-1,1)$ where the potential $U^\nu$ is continuous; as a contour $\Gamma$ in \eqref{eq:estimate} we take the level curve of $\Phi$, namely
	$$
	\Gamma = \Gamma_A \isdef  \left\{  z\in \C\setminus \Delta:\, |\Phi(z)|=\Phi(A) \right\}, \quad A>1. 
	$$
	It is easy to see that $\Gamma_A$ is an ellipse whose equation in Cartesian coordinates on the $(x,y)$ plane is
	$$
	\frac{x^2}{A^2} + \frac{y^2}{(A^2-1)}=1.
	$$
By \eqref{equilPot}, it is also a level curve of $U^{\mu_0}$.  Thus, the problem of minimizing $\log |z-\zeta|$ with $z\in \Gamma_A$ is equivalent to the real-valued constrained optimization problem
	\begin{align*}
		&\text{minimize} \quad  \left(x-\zeta\right)^2 + y^2 \\
		& \text{subject to} \quad \frac{x^2}{A^2} + \frac{y^2}{A^2-1}=1.
	\end{align*}
	Assume w.o.l.g.~that $0\le \zeta \le 1$. Standard methods show that the only possible points of minimum are $(x,y) =(A, 0)$ and, if $\zeta A^2\leq 1$, 
	$$
	(x,y)= \left( \zeta A^2, \pm \sqrt{\left(A^2-1 \right) \left( 1-\zeta \right)}\right).
	$$ 
	Hence, the global minimum is
	$$
	\begin{cases}
		(A-\zeta)^2, & \text{if } \zeta A^2\ge 1, \\
		\min \left\{ (A-\zeta)^2, (A^2-1) (1-\zeta^2 ) \right\} , & \text{if } \zeta A^2\le 1.
	\end{cases}
	$$
	Since
	$$
	(A-\zeta)^2 = (A^2-1) (1-\zeta^2 ) + (\zeta A-1)^2\ge (A^2-1) (1-\zeta^2 ),
	$$
	for $-1\le \zeta \le 1$ and $ z\in \Gamma_A$,
	$$
	\log|z-\zeta|\ge \frac{1}{2}\log (A^2-1) + \frac{1}{2}\log (1-\zeta^2), 
	$$
	so that, combining it with \eqref{equilPot},
	$$
	U^{\mu_0}(z)+t \log|z-\zeta| \ge  \log(2) - \log \left| \Phi(A) \right| + \frac{t}{2}\log (A^2-1) + \frac{t}{2}\log (1-\zeta^2).
	$$
	Using it in \eqref{eq:estimate}, we see that for $-1\le \zeta \le 1$,
	\begin{equation} \label{eq:lowerbound}
		U^{\mu_t}(\zeta) \ge -c_t +  \log(2) - \log \left| \Phi(A) \right| + \frac{t}{2}\log (A^2-1) + \frac{t}{2}\log (1-\zeta^2).
	\end{equation}
	With $\varphi$ as in \eqref{phi} we can rewrite it as
	$$
	\left( U^{\mu_t}+\varphi\right) (\zeta) \ge -c_t +  \log(2) - \log \left| \Phi(A) \right| + \frac{t}{2}\log (A^2-1) ,
	$$
	and we conclude that
	\begin{equation} \label{eq:lowerbound2}
		\min_{x\in [-1,1]}	\left( U^{\mu_0}+\varphi\right) (x) \ge -c_t +  \log(2) - \log \left| \Phi(A) \right| + \frac{t}{2}\log (A^2-1) .
	\end{equation}
	The global maximum of the right-hand side over $A>1$ is attained at
	$$
	A =\frac{1}{\sqrt{1-t^2}}=\frac{1}{s}
	$$
	and is equal to 
	$$
	-c_t +  \log(2) - \log \left| \Phi(1/s) \right| + \frac{t}{2}\log (s^{-2}-1) = -c_t +  \log(2) - \log \left| \Phi(1/s) \right| +t \log (t/s).
	$$
	Simplifying, we obtain that 
	\begin{align*}
		\min_{x\in [-1,1]}	\left( U^{\mu_t}+\varphi\right) (x) & \ge -c_t +  \log(2) + (1-t) \log
		(1-t)+\left(\frac{t}{2}-1\right)
		\log (2-t)+\frac{t}{2}  \log (t) 
	\end{align*}
	and \eqref{eq:newbound} follows.
\end{proof}
The assertion of Theorem~\ref{mainthm0} follows from Lemma~\ref{lem_lowerestimateArcsin} and 
the characterization~\eqref{charact1a}.

\section*{Acknowledgments}

The first author was partially supported by Simons Foundation Collaboration Grants for Mathematicians (grant 710499).
He also acknowledges the support of the project PID2021-124472NB-I00, funded by MCIN/AEI/10.13039/501100011033 and by ``ERDF A way of making Europe'', as well as the support of Junta de Andaluc\'{\i}a (research group FQM-229 and Instituto Interuniversitario Carlos I de F\'{\i}sica Te\'orica y Computacional).


\def\cprime{$'$}

\end{document}